\documentclass{article}
\usepackage{amsfonts}
\usepackage{amsmath}
\usepackage{amsthm}
\usepackage{amssymb}
\usepackage{bbm}
\usepackage{mathtools}
\usepackage{babel}
\usepackage{caption}
\usepackage{tikz}
\usepackage{tkz-graph}
\usepackage{geometry}
\numberwithin{equation}{section}
\numberwithin{figure}{section}
\theoremstyle{plain}
\newtheorem{theorem}{\protect\theoremname}[section]
\theoremstyle{definition}
\newtheorem{definition}[theorem]{\protect\definitionname}
\theoremstyle{definition}
\newtheorem*{definition*}{\protect\definitionname}
\theoremstyle{remark}
\newtheorem{remark}[theorem]{\protect\remarkname}
\theoremstyle{plain}
\newtheorem{lemma}[theorem]{\protect\lemmaname}
\theoremstyle{plain}
\newtheorem{corollary}[theorem]{\protect\corollaryname}
\theoremstyle{plain}
\newtheorem{prop}[theorem]{\protect\propositionname}
\theoremstyle{plain}

\newcounter{mysubequations}

\DeclarePairedDelimiter{\floor}{\lfloor}{\rfloor}

\DeclareMathOperator{\im}{Im }

\DeclareMathOperator{\vspan}{sp}

\DeclareMathOperator{\dist}{dist}

\DeclareMathOperator*{\argmax}{arg\,max}

\tikzstyle{vertex}=[circle, draw, fill=black, inner sep=0pt, minimum size=6pt]

\title{Universality of Packing Dimension Estimates for Spectral Measures of Quasiperiodic Operators: Monotone Potentials}
\author{Netanel Levi}
\makeatother
\date{\vspace{-5ex}}

\providecommand{\claimname}{Claim}
\providecommand{\corollaryname}{Corollary} 
\providecommand{\definitionname}{Definition}
\providecommand{\lemmaname}{Lemma}
\providecommand{\propositionname}{Proposition}
\providecommand{\remarkname}{Remark}
\providecommand{\theoremname}{Theorem}

\begin{document}
\maketitle
\sloppy
\begin{abstract}
	Let $H$ be a quasiperiodic Schr\"{o}dinger operator generated by a monotone potential, as defined in \cite{JK}. Following \cite{JLT}, we study the connection between the Lyapunov exponent $L\left(E\right)$, arithmetic properties of the frequency $\alpha$, and certain fractal-dimensional properties of the spectral measures of $H$.
\end{abstract}
\section{Introduction}
The purpose of this work is to complement the recent paper \cite{JLT} and study spectral continuity properties of quasiperiodic Schr\"{o}dinger operators acting on $\ell^2\left(\mathbb{Z}\right)$ which are generated by a $\gamma$-monotone potential. More specifically, we will study operators $H:\ell^2\left(\mathbb{Z}\right)\to\ell^2\left(\mathbb{Z}\right)$ of the form
\begin{equation}\label{eq_mono_op}
	\left(H\psi\right)\left(n\right)=\psi\left(n-1\right)+\psi\left(n+1\right)+f\left(x+n\alpha\right),
\end{equation}
where $x,\alpha\in\mathbb{R}$, and $f:\mathbb{R}\to\left[-\infty,\infty\right)$ is $1$-periodic and $\gamma$-monotone (see (\ref{eq_gamma_monotone}) for a precise definition). It was recently proved \cite{JK} that such operators exhibit a spectral phenomena called \textit{sharp arithmetic transition}, namely a change in the spectral type which depends on $L\left(E\right)$, the Lyapunov exponent, and $\beta\left(\alpha\right)$, a quantity related to the "measure of irrationality" of $\alpha$ (see \mbox{Section \ref{section_prelim}} for precise definitions). This phenomena was conjectured in in \cite{Jit} to hold in the case of the notable Almost Mathieu operator (AMO), given by
\begin{equation}\label{eq_AMO}
	\left(H\psi\right)\left(n\right)=\psi\left(n-1\right)+\psi\left(n+1\right)+2\lambda\cos\left(\theta+n\alpha\right),
\end{equation}
where $\theta$, $\alpha$ and $\lambda$ are called the \textit{phase}, \textit{frequency} and \textit{coupling constant} respectively. Sharp arithmetic transition was proved for AMO in \cite{AZY} in a measure-theoretic sense, namely for Lebesgue almost every phase. Then, in \cite{JW}, building on methods developed in \cite{Jit0} and subsequently in \cite{AJ,Jit2,LY}, the authors found a way of sharp analysis of frequency resonances which allowed them to obtain an arithmetic version of this transition, namely for a set of phases (of full Lebesgue measure) defined by a certain arithmetic condition. Using this method, the authors of \cite{JW} also determine the exact exponential asymptotics of the eigenfunctions. Sharp arithmetic transition was proved to be universal for other families of quasiperiodic operators. In \cite{JK2,JK}, the authors prove it for operators of the form (\ref{eq_mono_op}). Similarly to \cite{JW}, in \cite{JK} the authors prove localization by the method of dealing with frequency resonances, albeit using different techniques which are designed for operators of the form (\ref{eq_mono_op}), going back to \cite{JK2,Kach}. Sharp arithmetic transition was also proved to be universal for Type I operators \cite{GJ}, as well as for certain unitary analogues of the AMO \cite{CF,CFO,FY}.

Recently, this connection was taken further in \cite{JLT}, where the authors study certain fractal continuity properties of the spectral measures. These kind of properties are given, for example, by decomposing the spectral measures with respect to Hausdorff measures \cite{RT} or packing measures \cite{Cut}. It is well-known that positive Lyapunov exponent implies zero-dimensionality in the Hausdorff sense \cite{JL1}. However, much less is known about the packing dimension. Packing dimensional properties were studied in \cite{JZ}, where the authors prove lower bounds on the packing dimension, that are sharp in the case where $L\left(E\right)$ is much smaller than $\beta\left(\alpha\right)$. In \cite{JLT}, the authors develop general tools to systematically study such continuity properties, and subsequently apply these tools to study these properties for the AMO, near the transition. More precisely, they show that the packing dimension of the spectral measures is bounded from above by $2\left(1-\frac{L\left(E\right)}{\beta\left(\alpha\right)}\right)$. Roughly speaking, the authors develop partial decay estimates in the spirit of \cite{JW} for the case where $L\left(E\right)\leq\beta\left(\alpha\right)$. More precisely, even though in this case there is no uniform exponential decay of the generalized eigenfunctions, there are large intervals on which solutions have a localized behavior. The existence of such intervals leads to certain estimates of restricted $\ell^2$-norms of generalized eigenfunctions (see Proposition \ref{prop_green_omega_estimate}), which are known to be connected to fractal continuity properties of the spectral measures \cite{JL1,JL2,DKL}.

Even though in \cite{JLT} the authors focus on the AMO, it is mentioned there that the same techniques should also work for other quasiperiodic operators which exhibit a sharp arithmetic transition, and so these fractal-dimensional estimates in the singular continuous spectrum should be universal for such families. In this work, we exploit the techniques developed in \cite{JLT} along with decay estimates developed in \cite{JK} in order to verify this statement in the case of $\gamma$-monotone potentials.

Let us fix $\alpha\in\mathbb{R}$. Given $x\in\mathbb{R}$, we denote by $H\left(x\right)$ the operator given by (\ref{eq_mono_op}). For every $\varphi\in\ell^2\left(\mathbb{Z}\right)$, let $\mu_\varphi^x$ be the spectral measure of $\varphi$ w.r.t.\ $H\left(x\right)$ (see (\ref{eq_spec_meas_def})) and let $\mu^x=\mu^x_{\delta_0}+\mu^x_{\delta_1}$. Also let $\dim_P^+\left(\mu^x\right)$ denote its upper packing dimension (see Section \ref{section_dims}). Our main result is the following.

\begin{theorem}\label{main_thm_intro}
	For every Borel set $A\subseteq\mathbb{R}$, let $\mu^x|_A=\mu^x\left(\cdot\cap A\right)$. For every $x\in\mathbb{R}$, we have the following.
	\begin{enumerate}
		\item Suppose $L\left(E\right)\geq\beta\left(\alpha\right)$ for every $E\in A$. Then $\dim_P^+\left(\mu^x|_A\right)=0$.
		\item Suppose $L\left(E\right)<\beta\left(\alpha\right)$ for every $E\in A$ and let $L=\underset{E\in A}{\min}\,L\left(E\right)$. Then
		\begin{center}
			$\dim_P^+\left(\mu^x|_A\right)\leq 2\left(1-\frac{L}{\beta\left(\alpha\right)}\right)$.
		\end{center}
	\end{enumerate}
\end{theorem}
\begin{remark}
	\begin{enumerate}
		\item Theorem \ref{main_thm_intro} is the analogue of \cite[Theorem 1.1]{JLT} in the monotone case.
		\item It is well-known (see, e.g.\ \cite[Proposition 2.2.1]{DF1}) that $\left\{\delta_0,\delta_1\right\}$ is cyclic for $H\left(x\right)$ and consequently every spectral measure $\mu_\varphi^x$ is absolutely continuous w.r.t.\ $\mu^x$.
		\item Note that if $L\leq\frac{\beta\left(\alpha\right)}{2}$, the statement of Theorem \ref{main_thm_intro} is vacuous.
	\end{enumerate}
\end{remark}

Similarly to \cite{JLT}, we also establish bounds on the R\'{e}nyi dimension of $\mu^x$, $D_{\mu^x}^+\left(q\right)$, which also describes certain fractal continuity properties of $\mu^x$ (see Section \ref{section_dims} for a precise definition).
\begin{theorem}\label{main_thm_intro_2}
	Let $\alpha\in\mathbb{R}\setminus\mathbb{Q}$ and let $A\subseteq\mathbb{R}$ be a Borel set. Denote $L=\underset{E\in A}{\min}\,L\left(E\right)$. Then for every $q\geq\frac{3}{2}$ and for every $x\in\mathbb{R}$, we have
	\begin{center}
		$D_{\mu^x|_A}^+\left(q\right)\leq\frac{2\beta\left(\alpha\right)-2L}{2\beta\left(\alpha\right)-L}$.
	\end{center}
\end{theorem}
\begin{remark}
	Theorem \ref{main_thm_intro_2} is the analogue of \cite[Theorem 1.2]{JLT} in the monotone case.
\end{remark}

The rest of the paper is structured as follows. In Section \ref{section_prelim} we present some preliminary definitions and results related to quasiperiodic Schr\"{o}dinger operators. In Section \ref{section_green_monotone} we give relevant results from \cite{JK} which are crucial to obtain the results of this work, and we formulate Theorem \ref{main_thm} which then implies Theorems \ref{main_thm_intro} and \ref{main_thm_intro_2}. In Section \ref{section_proof_main_thm} we prove Theorem \ref{main_thm}.

{\bf Acknowledgments} I would like to thank Svetlana Jitomirskaya for suggesting the problem and for useful discussions. I would also like to thank Ilya Kachkovskiy, Wencai Liu and Jake Fillman for their helpful comments.
\section{Preliminaries}\label{section_prelim}
\subsection{Continued fraction approximation and $\beta\left(\alpha\right)$}
Let $\alpha\in\mathbb{R}\setminus\mathbb{Q}$. It is well-known that $\alpha$ has a continued fraction approximation, namely an approximation of the form $\left(\frac{p_n}{q_n}\right)_{n=0}^{\infty}$ where all $q_n$'s are positive, and there exists a sequence $\left(a_n\right)_{n=0}^\infty$ such that
\begin{center}
	$\frac{p_n}{q_n}=a_0+\frac{1}{a_1+\frac{1}{a_2+\frac{1}{\cdots+\frac{1}{a_n}}}}$.
\end{center}
Let $\beta\left(\alpha\right)\coloneqq\underset{n\to\infty}{\limsup}\,\frac{\ln q_{n+1}}{q_n}$. While $\beta\left(\alpha\right)$ can possibly be infinite, by the nature of the results in this paper we assume throughout that $\beta\left(\alpha\right)<\infty$. The number $\beta\left(\alpha\right)$ is often refered to as the measure of irrationality of $\alpha$. Note that by the definition of $\limsup$, for every $\varepsilon>0$, for large enough $n$,
\begin{center}
	$q_{n+1}\leq e^{\left(\beta\left(\alpha\right)+\varepsilon\right)q_n}$.
\end{center}
\subsection{Packing and R\'{e}nyi dimension of sets and measures}\label{section_dims}
\subsubsection{The packing dimension}
Fix $\delta>0$. Given a Borel set $S\subseteq\mathbb{R}$, a $\delta$-packing of $B$ is a collection of disjoint closed balls $\left(B\left(x_i,r_i\right)\right)_{i\in\mathbb{N}}$ such that for every $i\in\mathbb{N}$, $x_i\in B$ and $r_i\leq\delta$. For any $\alpha\in\left[0,1\right]$, the $\alpha$-dimensional packing pre-measure is defined by
\begin{center}
	$p_0^\alpha\left(S\right)=\underset{\delta\to 0}{\lim}\underset{\delta-\text{packings}}{\sup}\sum\limits_{i=1}^{\infty}\left|r_i\right|^\alpha$
\end{center}
for any $S\subseteq\mathbb{R}$. Then, the packing measure $p^\alpha$ is defined by
\begin{center}
	$p^\alpha\left(S\right)=\inf\left\{\sum\limits_{j=1}^
	\infty p_0^\alpha\left(S_j\right):S\subseteq\underset{j}{\bigcup}S_j\right\},\,\,\,\,\,\,\,S\in\text{Borel}\left(\mathbb{R}\right)$
\end{center}
For any Borel set $S$, there exists a unique $\alpha\in\left[0,1\right]$ such that $h^\beta\left(S\right)=\infty$ for any $\beta<\alpha$ and $h^\beta\left(S\right)=0$ for any $\beta>\alpha$. This $\alpha$ is called the {\it packing dimension} of $S$ and is denoted $\dim_P\left(S\right)$.

Let $\mu$ be a Borel measure on $\mathbb{R}$. The upper and lower packing dimensions of $\mu$ are defined by
\begin{center}
	$\dim_P^+\left(\mu\right)=\inf\left\{\dim_P\left(S\right):\mu\left(\mathbb{R}\setminus S\right)=0\right\}$,\\
	$\dim_P^-\left(\mu\right)=\inf\left\{\dim_P\left(S\right):\mu\left(S\right)>0\right\}$.
\end{center}
We also define, for every $E\in\mathbb{R}$,
\begin{center}
	$\gamma_\mu^-\left(E\right)=\underset{\varepsilon\to 0_+}{\liminf}\,\frac{\log\left(\mu\left(E-\varepsilon,E+\varepsilon\right)\right)}{\log\varepsilon}$,\\
	$\gamma_\mu^+\left(E\right)=\underset{\varepsilon\to 0_+}{\limsup}\,\frac{\log\left(\mu\left(E-\varepsilon,E+\varepsilon\right)\right)}{\log\varepsilon}$.
\end{center}
Finally, we define for every $E\in\mathbb{R}$ and $\eta\in\left[0,1\right]$,
\begin{center}
	$\underline{D}_\mu^\eta\left(E\right)=\underset{\varepsilon\to0}{\liminf}\,\frac{\mu\left(E-\varepsilon,E+\varepsilon\right)}{\varepsilon^\eta}$.
\end{center}

We will use the following results which can be found in the Appendix of \cite{GSB} (see also \cite{Cut}).
\begin{lemma}\label{packing_derivative_infinity_lemma}
	For every $\eta>\gamma_\mu^+\left(E\right)$, $\underline{D}_\mu^\eta\left(E\right)=\infty$.
\end{lemma}
\begin{prop}\label{prop_alpha_pack_supp}
	Let $S$ be a Borel set and let $\mu|_S\coloneqq\mu\left(\cdot\cap S\right)$. Suppose that for every $E\in S$, $\underline{D}_\mu^\eta\left(E\right)=\infty$. Then $\dim_P^+\left(\mu|_A\right)\leq\eta$.
\end{prop}
\subsubsection{The R\'{e}nyi dimension}
Let $\mu$ be a Borel measure on $\mathbb{R}$. For every $\varepsilon>0$, $q>0$, we denote
\begin{center}
	$S_\mu\left(q,\varepsilon\right)=\sum\limits_{j\in\mathbb{Z}}\left(\mu\left(\left[j\varepsilon,\left(j+1\right)\varepsilon\right)\right)\right)^q$.
\end{center}
The generalized R\'{e}nyi dimensions $D_\mu^\pm\left(q\right)$ are then defined by
\begin{center}
	$D_\mu^-\left(q\right)=\underset{\varepsilon\to0^+}{\liminf}\,\frac{\log S_\mu\left(q,\varepsilon\right)}{\left(q+1\right)\log\varepsilon}$,\\
	$D_\mu^+\left(q\right)=\underset{\varepsilon\to0^+}{\limsup}\,\frac{\log S_\mu\left(q,\varepsilon\right)}{\left(q+1\right)\log\varepsilon}$.
\end{center}
\subsection{Basic concepts in the theory of discrete Schr\"{o}dinger operators}

Let $H:\ell^2\left(\mathbb{Z}\right)\to\ell^2\left(\mathbb{Z}\right)$ be a bounded Schr\"{o}dinger operator,
\begin{equation}\label{op_eq}
	\left(H\psi\right)\left(n\right)=\psi\left(n-1\right)+\psi\left(n+1\right)+V\left(n\right)\psi\left(n\right)
\end{equation}
where $V:\mathbb{Z}\to\mathbb{R}$. Since $H$ is essentially self-adjoint \cite{Ber}, it corresponds with a projection-valued spectral measure $P$. For every $\varphi\in\ell^2\left(\mathbb{Z}\right)$, the spectral measure of $\varphi$ w.r.t.\ $H$ is a Borel measure given by
\begin{equation}\label{eq_spec_meas_def}
	\mu_\varphi\left(A\right)=\langle\varphi,P\left(A\right)\varphi\rangle,\,\,\,\,\,\,\,A\in\text{Borel}\left(\mathbb{R}\right).
\end{equation}
Since the set $\left\{\delta_0,\delta_1\right\}$ is cyclic for $H$, namely $\ell^2\left(\mathbb{Z}\right)=\overline{\vspan\left\{\left(H-z\right)^{-1}\delta_i:z\in\mathbb{C}\setminus\sigma\left(H\right),i=1,2\right\}}$, continuity properties of all spectral measures are determined by those of $\mu\coloneqq\mu_{\delta_0}+\mu_{\delta_1}$. Such properties are in turn strongly connected to its Borel transform, \mbox{$M\left(z\right)=\int_\mathbb{R}\frac{d\mu\left(x\right)}{x-z}$}, and to asymptotic properties of solutions to the eigenvalue equation. We now turn to describe the connections which are relevant for this work.
\subsubsection{The Borel transform}
Let $\mu$ be a finite real Borel measure on $\mathbb{R}$ and let $m:\mathbb{C}_+\to\mathbb{C}_+$ be its Borel transform,
\begin{center}
	$m\left(z\right)=\int_\mathbb{R}\frac{d\mu\left(x\right)}{x-z}$.
\end{center}
It is well-known that $m$ is analytic and that its boundary behavior is strongly connected to continuity properties of $\mu$ (for related results see, for example, \cite{DF1,DJLS,Sim1}). In particular, it is shown in \cite{DJLS} that for every \mbox{$\alpha\in\left(0,1\right)$} and every $E\in\mathbb{R}$, the quantities $\overline{D}_\mu^\alpha\left(E\right)$, $\underset{\varepsilon\to0}{\limsup}\,\varepsilon^{1-\alpha}\im m\left(E+i\varepsilon\right)$ and $\underset{\varepsilon\to0}{\limsup}\,\varepsilon^{1-\alpha}\left|m\left(E+i\varepsilon\right)\right|$ are either all $0$, all $\infty$, or all in $\left(0,\infty\right)$. The situation is more delicate for $\underline{D}_\mu^\alpha$. Nevertheless, we have the following results which are proved in \cite{JLT}.
\begin{prop}\emph{\cite[Theorem 1.3]{JLT}}\label{packing_dim_borel_transform_prop}
	Let $0\leq\eta<1$. Suppose that $\underset{\varepsilon\to 0}{\liminf}\,\varepsilon^{1-\eta}\im m\left(E+i\varepsilon\right)>0$. Then
	\begin{equation}
		\gamma_\mu^+\left(E\right)\leq\frac{\eta\left(2-\gamma_\mu^-\left(E\right)\right)}{2-\eta}.
	\end{equation}
	In particular, $\gamma_\mu^+\left(E\right)\leq\frac{2\eta}{2-\eta}$.
\end{prop}
\begin{prop}\emph{\cite[Theorem 1.7]{JLT}}\label{Renyi_dim_borel_transform_prop}
	Let $\gamma\in\left(0,1\right)$. Assume that there exists a Borel set $A\subseteq\mathbb{R}$ with $\mu\left(A\right)>0$ and such that for every $x\in A$,
	\begin{center}
		$\underset{\varepsilon\to0^+}{\liminf}\,\varepsilon^\gamma\im m\left(E+i\varepsilon\right)>0$.
	\end{center}
	Then for every $q\geq\frac{3}{2}$, $D_\mu^+\left(q\right)\leq\gamma$.
\end{prop}
\begin{remark}
	Propositions \ref{packing_dim_borel_transform_prop} and \ref{Renyi_dim_borel_transform_prop} are actually special cases of Theorems 1.3 and 1.7 in \cite{JLT}, which are more general.
\end{remark}
\subsubsection{Half-line restrictions of Schr\"{o}dinger operators}
For every $A\subseteq\mathbb{Z}$, let $P_A$ denote the projection of $\ell^2\left(\mathbb{Z}\right)$ onto $\ell^2\left(A\right)$. Let
\begin{center}
	$H^\pm\coloneqq P_{\mathbb{Z}_\pm}HP_{\mathbb{Z}_\pm}$,
\end{center}
where $\mathbb{Z}_+=\mathbb{N}$ and $\mathbb{Z}_-=\mathbb{Z}\setminus\mathbb{N}$. Given $\theta\in\left[0,\pi\right)$, let
\begin{center}
	$H_\theta^+=\begin{cases}
		H^+-\tan\theta\langle\delta_1,\cdot\rangle\delta_1 & \theta\neq\frac{\pi}{2}\\
		H_{\frac{\pi}{2}}^+ & \theta=\frac{\pi}{2}
	\end{cases}$,\\
	$H_\theta^-=\begin{cases}
		H^--\cot\theta\langle\delta_0,\cdot\rangle\delta_0 & \theta\neq 0\\
		H_0^- & \theta=0
	\end{cases}$,
\end{center}
where $H_\frac{\pi}{2}^+$ is defined by shifting $V$ to the left and taking the corresponding $H^+$, and $H_0^-$ is defined by shifting $V$ to the right and taking the corresponding $H^-$. Note that for every $\theta\in\left[0,\pi\right)$, $\delta_0$ and $\delta_1$ are cyclic vectors for $H_\theta^-$ and $H_\theta^+$ respectively. We denote their spectral measures by $\mu_\theta^\pm$ and the corresponding Borel transforms by $m_\theta^\pm$. Throughout this work, we will also denote by $\mu_0,\mu_1$ the spectral measures of $\delta_0$ and $\delta_1$ w.r.t.\ $H$, and by $M_0,M_1$ their respective Borel transforms. Finally, we will denote $\mu=\mu_0+\mu_1$ and $M=M_0+M_1$. It is well-known (see, e.g.\ \cite[(2.5)-(2.7)]{JL2}) that for every $z\in\mathbb{C}_+$,
\begin{align*}\label{M_rep_using_hl_eq}
	M_0\left(z\right)=\frac{m_0^+\left(z\right)m_0^-\left(z\right)}{m_0^+\left(z\right)+m_0^-\left(z\right)},\\
	M_1\left(z\right)=\frac{-1}{m_0^+\left(z\right)+m_0^-\left(z\right)}
\end{align*}
which implies
\begin{center}
	$M\left(z\right)=M_0\left(z\right)+M_1\left(z\right)=\frac{m_0^+\left(z\right)m_0^-\left(z\right)-1}{m_0^+\left(z\right)+m_0^-\left(z\right)}$.
\end{center}
We will use the following result from \cite{JLT}.
\begin{prop}\emph{\cite[Lemma 9.1]{JLT}}\label{line_borel_hl_borel_prop}
	Let $E\in\mathbb{R}$ and let $t\in\left(0,1\right)$. Suppose that there exist $\theta\in\left[0,\pi\right)$ and $\varepsilon_0>0$ such that for $\varepsilon<\varepsilon_0$, $\im m_\theta^\pm\left(E+i\varepsilon\right)\geq e^{-t}$. Then there exists $\varepsilon_1>0$ such that for $\varepsilon<\varepsilon_1$,
	\begin{equation}
		\im M\left(E+i\varepsilon\right)\geq\varepsilon^{-t}.
	\end{equation}
\end{prop}
\subsubsection{Asymptotic properties of generalized eigenfunctions}
Recall that $H$ is a Schr\"{o}dinger operator given by (\ref{op_eq}). In this subsection, we discuss some properties of solutions to the eigenvalue equation. Namely, given $E\in\mathbb{R}$, we are concerned with fucntions $u:\mathbb{Z}\to\mathbb{R}$ which satisfy
\begin{equation}\label{ev_eq_line_op}
	u\left(n+1\right)+u\left(n-1\right)+V\left(n\right)u\left(n\right)=Eu\left(n\right)
\end{equation}
for all $n\in\mathbb{Z}$. It is well-known (see, e.g.\ \cite{Ber}) that for $\mu$-almost every $E\in\mathbb{R}$ there exists a solution $u_E$ to (\ref{ev_eq_line_op}) which satisfies
\begin{equation}\label{gen_ev_estimate_eq}
	\left|\psi_E\left(n\right)\right|\leq C_E\left(1+\left|n\right|\right)
\end{equation}
for every $n\in\mathbb{Z}$. We call a solution $\phi$ which satisfies (\ref{gen_ev_estimate_eq}) and in addition satisfies
\begin{equation}\label{eq_norm_gen_ef}
	\left|\phi\left(0\right)\right|^2+\left|\phi\left(1\right)\right|^2=1
\end{equation}
a \textit{generalized eigenfunction}.

A connection between "truncated" $\ell^2$-norms and singularity was also obtained. More precisely, given $L>0$ and $u:\mathbb{Z}\to\mathbb{C}$, let
\begin{center}
	$\|u\|_L^+\coloneqq\left(\sum\limits_{n=1}^{\floor{L}}\left|u\left(n\right)\right|^2+\left(L-\floor{L}\right)\left|u\left(\floor{L}+1\right)\right|^2\right)^{\frac{1}{2}}$,\\
	$\|u\|_L^-\coloneqq\left(\sum\limits_{n=0}^{\floor{L}}\left|u\left(-n\right)\right|^2+\left(L-\floor{L}\right)\left|u\left(-\floor{L}-1\right)\right|^2\right)^{\frac{1}{2}}$.
\end{center}
It was proved in \cite{LS} that for $\mu$-almost every $E\in\mathbb{R}$ there exist $C\left(E\right)>0$ and solution $u$ to (\ref{ev_eq_line_op}) which satisfies
\begin{equation}\label{eq_last_simon}
	\|u\|_L^-+\|u\|_L^+\leq C\left(E\right)L^{\frac{1}{2}}\ln L.
\end{equation}
For every $\theta\in\left[0,\pi\right)$, there exists a unique solution $u_\theta$ to (\ref{ev_eq_line_op}) which satisfies
\begin{equation}\label{eq_bcon}
	u_\theta\left(0\right)=-\sin\theta,\,\,\, u_\theta\left(1\right)=\cos\theta.
\end{equation}
It is well-known (see, e.g.\ \cite{JL1}) that given $\theta\in\left[0,\pi\right)$, if we denote $\widetilde{\theta}=\theta+\frac{\pi}{2}\mod\pi$, then the Wronskian is constant, namely for every $n\in\mathbb{Z}$,
\begin{equation}\label{eq_const_wronskian}
	u_\theta\left(n\right)u_{\widetilde{\theta}}\left(n+1\right)-u_\theta\left(n+1\right)u_{\widetilde{\theta}}\left(n\right)=1.
\end{equation}
Subordinacy \cite{GP,KP} and power-law subordinacy \cite{DKL,JL1,JL2} relate asymptotic properties of the $L$-norms of solutions to (\ref{ev_eq_line_op}) to the boundary behavior of the Borel transform of $\mu$. For every $L>0$, let
\begin{center}
	$\omega_\pm\left(L\right)\coloneqq\underset{\theta\in\left[0,\pi\right)}{\max}\|u_\theta\|_L^\pm\cdot\underset{\theta\in\left[0,\pi\right)}{\min}\|u_\theta\|_L^\pm$.
\end{center}
It is not hard to see that for every $\varepsilon>0$ there exist unique $L_\pm\left(\varepsilon\right)$ such that $\omega_\pm\left(L_\pm\left(\varepsilon\right)\right)=\frac{1}{\varepsilon}$. We will use the following theorem, which appears in different forms in \cite{DT,KKL}. See for example \mbox{\cite[Appendix A]{KKL}}.
\begin{theorem}\label{Appendix_thm}
	Fix $\theta\in\left[0,\pi\right)$ and let $\widetilde{\theta}=\theta+\frac{\pi}{2}\mod\pi$. Given $L>0$, let
	\begin{center}
		$b_\pm\left(L\right)\coloneqq\left(\|u_\theta\|_L^\pm\right)^2$.
	\end{center}
	Then there exists a constant $C$, independent of $\theta$, such that
	\begin{center}
		$\im m_\theta^\pm\left(E+i\varepsilon\right)\geq\frac{1}{C}\frac{1}{\varepsilon}\frac{1}{b_\pm\left(L\left(\varepsilon\right)\right)}$.
	\end{center}
\end{theorem}

\subsubsection{The Lyapunov exponent}
Fix $E\in\mathbb{R}$. For every $n\in\mathbb{Z}$, let
\begin{center}
	$T_n\left(E\right)=\left(\begin{matrix}
		E-V\left(n\right) & -1\\
		1 & 0
	\end{matrix}\right)$.
\end{center}
It is well known (see, e.g.\ \cite[I.1.10]{DF2}) that if $u$ is a solution to (\ref{ev_eq_line_op}), then for every $n\in\mathbb{Z}$,
\begin{center}
	$\left(\begin{matrix}
		u\left(n+1\right)\\
		u\left(n\right)
	\end{matrix}\right)=T_n\left(E\right)\left(\begin{matrix}
	u\left(n\right)\\u\left(n-1\right)
\end{matrix}\right)$.
\end{center}
Given $n\in\mathbb{Z}$, the $n$-step transfer matrix $\Phi_n\left(E\right)$ is then given by
\begin{center}
	$\Phi_n\left(E\right)=\begin{cases}
		T_n\left(E\right)\cdot\cdots\cdot T_1\left(E\right) & n\geq 1\\
		\text{Id} & n=0\\
		T_{n+1}^{-1}\cdot\cdots\cdot T_{0}^{-1} & n\leq -1
	\end{cases}.$
\end{center}
It is not hard to see that for every $n\in\mathbb{Z}$,
\begin{center}
	$\left(\begin{matrix}
		u\left(n+1\right)\\u\left(n\right)
	\end{matrix}\right)=\Phi_n\left(E\right)\left(\begin{matrix}
	u\left(1\right)\\u\left(0\right)
\end{matrix}\right)$.
\end{center}
Suppose that $V$ is a quasiperiodic monotone potential and fix $\alpha\in\left[0,1\right)\setminus\mathbb{Q}$. The transfer matrices depend on the choice of $x\in\left[0,1\right]$ and so we write $\Phi_n\left(x,E\right)$ to denote the $n$-step transfer matrix associated with $E$ and $H\left(x\right)$. The \textit{Lyapunov exponent} is then given by
\begin{center}
	$L\left(E\right)\coloneqq\underset{\left|n\right|\to\infty}{\lim}\frac{1}{n}\int_{\left[0,1\right)}\ln\|\Phi_n\left(x,E\right)\|dx$.
\end{center}
The Lyapunov exponent is well-defined and finite \cite{Kach}.
\section{Green's function estimate for monotone potentials}\label{section_green_monotone}
The goal of this section is to formally introduce the notion of quasiperiodic operators with monotone potentials, formulate our main result, and present results of \cite{JK} along with some relevant corollaries.
\subsection{The setting and main results}
Let $f:\mathbb{R}\to\left[-\infty,\infty\right)$ be a $1$-periodic function which satisfies
\begin{equation}\label{eq_gamma_monotone}
	f\left(y\right)-f\left(x\right)\geq\gamma\left(y-x\right),\,\,\,\,\,0\leq x<y<1
\end{equation}
for $\gamma>0$. Such $f$ is called \textit{$\gamma$-monotone}. Throughout, we will assume that
\begin{center}
	$\int_0^1\log\left(1+\left|f\left(x\right)\right|\right)dx<\infty$.
\end{center}
Given $\alpha,x\in\mathbb{R}$, the operator $H_{\alpha,x}:\ell^2\left(\mathbb{Z}\right)\to\ell^2\left(\mathbb{Z}\right)$ is given by
\begin{center}
	$\left(H_{\alpha,x}\psi\right)\left(n\right)=\psi\left(n-1\right)+\psi\left(n+1\right)+f\left(x+n\alpha\right)\psi\left(n\right)$.
\end{center}

Let us now fix $\alpha$ and study continuity properties of the measure $\mu^x=\mu_{\delta_0}^x+\mu_{\delta_1}^x$, where $\mu_\varphi^x$ is the spectral measure of $\varphi$ w.r.t.\ $H\left(x\right)=H\left(\alpha,x\right)$. Given $E\in\mathbb{R}$ and $\eta\in\left[0,1\right]$, recall that the lower $\eta$-derivative of $\mu^x$ at $E$ is given by $\underline{D}_{\mu^x}^\eta\left(E\right)=\underset{\varepsilon\to0}{\liminf}\,\frac{\mu^x\left(E-\varepsilon,E+\varepsilon\right)}{\varepsilon^\eta}$. Our main result in this section is the following
\begin{theorem}\label{main_thm}
	Suppose $\alpha\in\mathbb{R}\setminus\mathbb{Q}$. For every $E\in\mathbb{R}$, denote $\Lambda\left(E\right)=\min\left\{L\left(E\right),\beta\left(\alpha\right)\right\}$. Then for every \mbox{$\eta>2\left(1-\frac{\Lambda\left(E\right)}{\beta\left(\alpha\right)}\right)$} and every $x\in\mathbb{R}$, for $\mu^x$-almost every $E\in\mathbb{R}$,
	\begin{equation}\label{packing_derivative_main_thm_eq}
		\underline{D}_{\mu^x}^\eta\left(E\right)=\infty.
	\end{equation}
\end{theorem}
\begin{remark}
	\begin{enumerate}
		\item If $L\left(E\right)>\beta\left(\alpha\right)$, then $\mu$-almost every $E\in\sigma\left(H\left(x\right)\right)$ must be an atom of $\mu^x$ and so (\ref{packing_derivative_main_thm_eq}) holds for every $\eta>0$.
		\item Note that if $L\left(E\right)<\frac{\beta\left(\alpha\right)}{2}$, then (\ref{packing_derivative_main_thm_eq}) holds for every Borel measure since $\eta>1$.
	\end{enumerate}
\end{remark}

\begin{proof}[Proof of Theorem \ref{main_thm_intro}]
	This is a direct consequence of Theorem \ref{main_thm} and Proposition \ref{prop_alpha_pack_supp}.
\end{proof}

\begin{proof}[Proof of Theorem \ref{main_thm_intro_2}]
	This is a direct consequence of Theorem \ref{main_thm} and Proposition \ref{Renyi_dim_borel_transform_prop}.
\end{proof}
\subsection{Green's function and decay estimates of solutions to (\ref{ev_eq_line_op})}
\subsubsection{Green's function}
Let $H:\ell^2\left(\mathbb{Z}\right)\to\ell^2\left(\mathbb{Z}\right)$ be given by (\ref{op_eq}). Given $\left[n_1,n_2\right]=I\subseteq\mathbb{Z}$, let $H_I$ be the restriction of $H$ to the interval $I$, namely
\begin{center}
	$H_I=P_IHP_I$.
\end{center}
Green's function is then given by
\begin{center}
	$G_I\left(x,y\right)=P_I\left(H_I-E\right)^{-1}P_I\left(x,y\right)$, $x,y\in I$.
\end{center}
Note that $G_I\left(x,y\right)=G_I\left(y,x\right)$ for every $E\in\mathbb{R}$. It is not hard to see (see, e.g.\ \cite{DF1}) that given a solution $u$ to (\ref{ev_eq_line_op}), for every $n\in\mathbb{Z}$ and every interval $\left[n_1,n_2\right]=I$ containing $n$,
\begin{equation}\label{sol_exp_green_fn_eq}
	u\left(n\right)=-G_I\left(n_1,n\right)u\left(n_1-1\right)-G_I\left(n,n_2\right)u\left(n_2+1\right).
\end{equation}
Thus, estimates on Green's function lead to estimates on the growth rate of solutions.
\begin{definition}
	Let $t>0$ and $k\in\mathbb{N}$. A point $n\in\mathbb{Z}$ is called $\left(t,k\right)$-regular if there exists an interval $\left[n_1,n_2\right]=I\subseteq\mathbb{Z}$ such that $n\in I$, $x_2-x_1=k-1$, $n_2-n,n-n_1\geq\frac{k}{2}$ and
	\begin{center}
		$\left|G_I\left(n,n_i\right)\right|\leq e^{-t\left|n-n_i\right|}$.
	\end{center}
	If $n$ is $\left(t,k\right)$-regular, then we denote $I\left(n\right)=\left[n_1,n_2\right]$.
\end{definition}
Using iterative estimations of solutions using (\ref{sol_exp_green_fn_eq}), the following is obtained in \cite{JK}.
\begin{prop}\label{inner_point_estimate_prop}
	Let $n\in\mathbb{Z}$ and suppose that there exist intervals $\left[n_1,n_2\right]\subset\left[N_1,N_2\right]\subset\mathbb{Z}$, $t>0$ and $k\in\mathbb{N}$ such that
	\begin{enumerate}
		\item Every $n\in\left[n_1,n_2\right]$ is $\left(\psi,t_n,k_n\right)$-regular, with the corresponding interval contained in $\left[N_1,N_2\right]$.
		\item For every $n\in\left[n_1,n_2\right]$, $t_n\geq t$ and $k_n>k$.
	\end{enumerate}
	Denote
	\begin{center}
		$N_1'=\argmax{\left\{\left|\psi\left(N\right)\right|:N\in\left[N_1,n_1\right)\right\}}$,\\
		$N_2'=\argmax{\left\{\left|\psi\left(N\right)\right|:N\in\left(n_2,N_2\right]\right\}}$.
	\end{center}
	Then for every $n\in\left[n_1,n_2\right]$,
	\begin{equation}\label{inner_point_estimate_eq}
		\left|\psi\left(n\right)\right|\leq e^{-t\left(1-o\left(1\right)\right)\left|m-N_1\right|}\left|\psi\left(N_1'\right)\right|+e^{-t\left(1-o\left(1\right)\right)\left|m-N_2\right|}\left|\psi\left(N_2'\right)\right|.
	\end{equation}
\end{prop}
\begin{remark}
	\begin{enumerate}
		\item Proposition \ref{inner_point_estimate_prop} did not appear in \cite{JK} in this precise form. However, (\ref{inner_point_estimate_eq}) is precisely (5.10) there.
		\item Here, $o\left(1\right)$ is as $k\to\infty$.
	\end{enumerate}
\end{remark}
\subsubsection{Regularity estimates in certain regions}
Throughout this section, we fix some constant $C>0$. Recall that $\left(q_n\right)_{n=1}^\infty$ is the sequence of denominators in the continued fraction approximation of $\alpha$. In the results of this section, $o\left(1\right)$ will always be as $n\to\infty$.
\begin{prop}\emph{\cite[Proposition 5.6]{JK}}\label{prop_diophantine_transition}
	Assume that $q_{n+1}\leq q_n^C$. Then, every point $m\in\mathbb{Z}$ with
	\begin{center}
		$\left|m\right|\in\left[\floor{\frac{q_n}{2}+1},q_{n+1}-\floor{\frac{q_n}{2}}-1\right]$
	\end{center}
	is $\left(L\left(E\right)-o\left(1\right),q_n\right)$-regular.
\end{prop}
Let $\tau>0$ be some small parameter. Following the technique introduced in \cite{LY1}, let us denote by $n_0\in\mathbb{N}$ be the least positive integer such that $q_{n-n_0}\leq\tau q_n$, and subsequently let $s>0$ be the largest such that $2sq_{n-n_0}\leq\tau q_n$. Also let $b_n\coloneqq\floor{\tau q_n}$. Given $l\in\mathbb{Z}$, let $R_l=\left[lq_n-b_n,lq_n+b_n\right]$. Given a solution $u$ to (\ref{ev_eq_line_op}), let $r_l^u\coloneqq\underset{x\in R_l}{\max}\left|u\left(x\right)\right|$.
\begin{definition}
	A point $k\in\left[-q_{n+1},q_{n+1}\right]$ is called \textit{$n$-resonant} if $k\in\underset{l\in\mathbb{Z}}{\bigcup} R_l$. Otherwise, $k$ is called \textit{$n$-nonresonant}
\end{definition}
We will use the following result, which establishes regularity of points that are non-resonant, in the case where $q_{n+1}>q_n^C$.
\begin{prop}\emph{\cite[Lemmas 5.10 and 5.11]{JK}}\label{JK_Liouville_Lemmas}
	Suppose that $q_{n+1}>q_n^C$.
	\begin{enumerate}
		\item If $s\leq q_{n-n_0}^C$, then every $n$-nonresonant point $n\in\left[-q_{n+1},q_{n+1}\right]$ is $\left(L\left(E\right)-o\left(1\right),2sq_{n-n_0}-1\right)$-regular.
		\item If $s\geq q_{n-n_0}^C$, then every $n$-nonresonant point $n\in\left[-q_{n+1},q_{n+1}\right]$ is $\left(L\left(E\right)-o\left(1\right),2s'q_{n-n_0}-1\right)$-regular, where $s'=\floor{\frac{s}{10}}$.
	\end{enumerate}
\end{prop}
Proposition \ref{JK_Liouville_Lemmas} and (\ref{inner_point_estimate_eq}) have the following immediate corollary.
\begin{corollary}\label{cor_growth_between_resonances}
	Suppose that $q_{n+1}>q_n^C$. Then for every $j\in\mathbb{Z}$ and every \mbox{$k\in\left[jq_n,\left(j+1\right)q_n\right]\cap\left[-q_{n+1},q_{n+1}\right]$} which is $n$-nonresonant,
	\begin{equation}
		\left|\psi\left(k\right)\right|\leq e^{-\left(L\left(E\right)-o\left(1\right)\right)\left(1-o\left(1\right)\right)\left|k-jq_n+b_n\right|}r_j^\psi+e^{-\left(L\left(E\right)-o\left(1\right)\right)\left(1-o\left(1\right)\right)\left|k-\left(j+1\right)q_n+b_n\right|}r_{j+1}^\psi.
	\end{equation}
\end{corollary}
We will also need the following Proposition, which connects regularity and asymptotic properties of solutions discussed in the previous section. Recall that
\begin{center}
	$\omega_\pm\left(L\right)=\left(\underset{\theta\in\left[0,\pi\right)}{\min}\|u_\theta\|_L^\pm\right)\left(\underset{\theta\in\left[0,\pi\right)}{\max}\|u_\theta\|_L^\pm\right)$.
\end{center}
\begin{prop}\label{prop_green_omega_estimate}
	Fix $L>0$. Suppose that $k\in\mathbb{N}$ is such that $k,k+1\in\left[1,L\right]$, $k,k+1$ are both $\left(\mu,n\right)$ regular and $I\left(k\right),I\left(k+1\right)\subseteq\left[1,L\right]$. Then
	\begin{equation}\label{omega_lower_bd}
		\omega_+\left(L\right)>\frac{e^{\frac{\mu n}{2}}}{4}.
	\end{equation}
\end{prop}
\begin{proof}
	Fix $\theta_1$ to be such that $\left\|u_{\theta_1}\right\|_L^+=\underset{\theta}{\min}\left\|u_\theta\right\|_L^+$. Also let $r=\underset{j\in\left[1,L\right]}{\max}\left|u_{\theta_1}\left(j\right)\right|$. By regularity and (\ref{sol_exp_green_fn_eq}), we have
	\begin{equation}
		\left|u_{\theta_1}\left(k\right)\right|,\left|u_{\theta_1}\left(k+1\right)\right|\leq2re^{-\frac{\mu n}{2}}.
	\end{equation}
	By constancy of the Wronskian (\ref{eq_const_wronskian}), we get for $\theta_2=\theta_1+\frac{\pi}{2}\mod\pi$
	\begin{equation}
		\max\left\{\left|u_{\theta_1}\left(k\right)\right|,\left|u_{\theta_2}\left(k+1\right)\right|\right\}\geq \frac{e^{\frac{\mu n}{2}}}{4r}
	\end{equation}
	which implies that $\left(\left\|u_{\theta_2}\right\|_L^+\right)^2\geq \frac{e^{\mu n}}{16r^2}$. In addition, $\left(\left\|u_{\theta_1}\right\|_L^+\right)^2\geq r^2$. Thus, we obtain
	\begin{center}
		$\omega_+\left(L\right)^2\geq\left(\left\|u_{\theta_2}\right\|_L^+\right)^2\left(\left\|u_{\theta_1}\right\|_L^+\right)^2\geq\frac{e^{\mu n}}{16r^2}r^2=\frac{1}{16}e^{\mu n}$,
	\end{center}
	which implies (\ref{omega_lower_bd}).
\end{proof}
\begin{remark}
	Although Proposition \ref{prop_green_omega_estimate} is formulated and proved for $\omega_+$, it is clear that it also holds for $\omega_-$.
\end{remark}
We will use the following lemma, which is essentially a special case of \cite[Theorem 6.4]{JK}. A proof is given in the Appendix.
\begin{lemma}\label{JK_Sec_6_Lemma}
	 Let $n\in\mathbb{N}$, $0\neq j\in\mathbb{Z}$. Set $I_1,I_2\subseteq\mathbb{Z}$ as follows:
	\begin{center}
		$I_1=\left[-\floor{\frac{3q_n}{2}},q_n-\floor{\frac{3q_n}{2}}-1\right]$,\\
		$I_2=\left[jq_n-\floor{\frac{3q_n}{2}},\left(j+1\right)q_n-\floor{\frac{3q_n}{2}}-1\right]$.
	\end{center}
	Then there exists $j_0\in I_1\cup I_2$ such that for $J=\left[j_0,j_0+2q_n-2\right]\coloneqq\left[n_1,n_2\right]$, for every $s\in J$ such that 
	\begin{equation}\label{eq_s_cond}
		j_0+\frac{1}{5}\left(2q_n-2\right)\leq s\leq j_0+2q_n-2-\frac{1}{5}\left(2q_n-2\right)
	\end{equation}
	we have
	\begin{equation}\label{eq_JK_sec_6}
		\left|G_{J}\left(s,n_i\right)\right|\leq\frac{e^{-\left|s-n_i\right|L\left(E\right)\left(1-o\left(1\right)\right)}}{\frac{d}{2}},
	\end{equation}
	where $d\geq e^{-\left(\beta+o\left(1\right)\right)q_n+\log\left|j\right|}$.
\end{lemma}
\begin{remark}
	In \cite{JK}, localization is proved assuming that $L\left(E\right)>\beta\left(\alpha\right)$. However, the results which we state here do not require this and are true also when $L\left(E\right)\leq\beta\left(\alpha\right)$.
\end{remark}
\section{Proof of Theorem \ref{main_thm}}\label{section_proof_main_thm}
Fix $t_2$ such that $\frac{9\beta-L\left(E\right)}{9\beta}<t_2<1$. Let $\sigma>0$ be small enough so that
\begin{equation}\label{eq_t1_t2_def}
	t_1\coloneqq\frac{\beta-L\left(E\right)}{\beta}+\sigma<t_2<1.
\end{equation}
We will use the results presented in the previous section with $C=\frac{4}{t_2-t_1}$ and $\tau>0$ some fixed small constant. Following the previous section, Let $b_n=\floor{\tau q_n}$ and recall that for every $j\in\mathbb{Z}$, $R_j=\left[jq_n-b_n,jq_n+b_n\right]$.
\begin{prop}\label{prop_liouville_resonant_growth}
	Suppose $k\in\mathbb{Z}$ is $n$-resonant and satisfies $2q_n^2q_{n+1}^{t_1}<\left|k\right|<q_{n+1}^{t_2}$. Let $\phi$ be a generalized eigenfunction. Then we have
	\begin{equation}\label{eq_phi_k_estimate}
		\left|\phi\left(k\right)\right|\leq\exp\left(-\frac{1}{2}\left(L\left(E\right)-\left(1-t_1\right)\beta-o\left(1\right)\right)\left|k\right|\right).
	\end{equation}
\end{prop}
The proof of Proposition \ref{prop_liouville_resonant_growth} follows the lines of \cite[Lemma 7.1]{JLT} with slight adaptations, since we use slightly different estimates.
\begin{proof}
	By our assumptions on $k$, there exist $l,r\in\mathbb{Z}$ such that $k=lq_n+r$, $q_nq_{n+1}^{t_1}\leq\left|l\right|\leq2\frac{q_{n+1}^{t_2}}{q_n}$ and $\left|r\right|\leq b_n$. For simplicity, we assume that $l>0$. Let $j\in\mathbb{Z}$ such that $\floor{\frac{l}{q_n}}\leq\left|j\right|\leq2l$. Denote $r_j=\underset{\left|r\right|\leq b_n}{\max}\left|\phi\left(jq_n+r\right)\right|$. Let us first show that
	\begin{equation}\label{eq_r_j_estimate}
		r_j\leq\max\left\{r_{j\pm k}\exp\left(-\left(L\left(E\right)-\left(1-t_1\right)\beta-o\left(1\right)\right)q_n\right):k\in\left\{1,2\right\}\right\},
	\end{equation}
	where $o\left(1\right)$ is as $n\to\infty$. Let $I_1,I_2\subseteq\mathbb{Z}$ be given by
	\begin{center}
		$I_1=\left[-\floor{\frac{3q_n}{2}},q_n-\floor{\frac{3q_n}{2}}-1\right]$,\\
		$I_2=\left[jq_n-\floor{\frac{3q_n}{2}},\left(j+1\right)q_n-\floor{\frac{3q_n}{2}}-1\right]$
	\end{center}
	and let $j_0,n_1,n_2$ be given by Lemma \ref{JK_Sec_6_Lemma}. It is not hard to see that there exists $p\in\mathbb{Z}$ such that $\left[pq_n-b_n,pq_n+b_n\right]\subseteq\left[n_1,n_2\right]$, and for every \mbox{$s\in\left[pq_n-b_n,pq_n+b_n\right]$}, (\ref{eq_s_cond}) holds. Thus, by Lemma \ref{JK_Sec_6_Lemma}, for every such $s$, (\ref{eq_JK_sec_6}) holds. Noting that $\left|j\right|\geq\floor{\frac{l}{q_n}}$ and $l\geq q_nq_{n+1}^{t_1}$, we obtain
	\begin{center}
		$\left|G_{J}\left(s,n_i\right)\right|\leq\frac{e^{-\left|s-n_i\right|L\left(E\right)\left(1-o\left(1\right)\right)}}{\frac{d}{2}}\leq e^{\left(\left(1-t_1\right)\beta+o\left(1\right)\right)q_n}e^{-\left|pq_n+r-n_i\right|L\left(E\right)}$
	\end{center}
	for $i=1,2$. Now, using the expansion formula (\ref{sol_exp_green_fn_eq}), we obtain
	\begin{equation}\label{eq_1_qn+r}
		\left|\phi\left(pq_n+r\right)\right|\leq\sum\limits_{i=1,2}e^{\left(\left(1-t_1\right)\beta+o\left(1\right)\right)q_n}\left|\phi\left(n_i'\right)\right|e^{-\left|pq_n+r-n_i\right|L\left(E\right)}
	\end{equation}
	where $n_i'=n_i+\left(-1\right)^i$. Now, let us analyze the factor $\left|\phi\left(n_i'\right)\right|e^{-\left|pq_n+r-n_i\right|L\left(E\right)}$ by splitting into several cases depending on the location of $n_i'$.
	\begin{enumerate}
		\item $\text{dist}\left(n_i',q_n\mathbb{Z}\right)\leq b_n$. In that case, $n_i'\in R_q$ for some $q$ and so $\left|\phi\left(n_i'\right)\right|\leq r_q$. Furthermore, by the definition of $n_i$ we have $q\in\left\{p-1,p+1\right\}$. This also implies that $\left|pq_n+r-n_i\right|\geq q_n-2b_n$. In that case, we get
		\begin{equation}
			\left|\phi\left(n_i'\right)\right|e^{-\left|pq_n+r-n_i\right|L\left(E\right)}\leq r_qe^{\left(q_n-2r\right)L\left(E\right)}\leq e^{\left(L\left(E\right)-o\left(1\right)\right)q_n}r_q.
		\end{equation}
		\item $\text{dist}\left(n_i',q_n\mathbb{Z}\right)\geq b_n$. Here, we split into cases again.
		\begin{enumerate}
			\item $n_i\in\left[\left(p-2\right)q_n,\left(p-1\right)q_n\right]$. In that case, by Corollary \ref{cor_growth_between_resonances}, we get
			\begin{center}
				$\left|\phi\left(n_i'\right)\right|e^{-\left|pq_n+r-n_i\right|L\left(E\right)}\leq \left(e^{-\left(L\left(E\right)-o\left(1\right)\right)\left(1-o\left(1\right)\right)\left|n_i'-\left(p-2\right)q_n+b_n\right|}r_{p-2}+e^{-\left(L\left(E\right)-o\left(1\right)\right)\left(1-o\left(1\right)\right)\left|n_i'-\left(p-1\right)q_n+b_n\right|}r_{p-1}\right)e^{-\left|pq_n+r-n_i\right|L\left(E\right)}$.
			\end{center}
			Now, note that
			\begin{center}
				$\left|pq_n+r-n_i\right|+\left|x_i'-\left(p-2\right)q_n+b_n\right|\geq 2q_n-2\left|b_n\right|$,\\
				$\left|pq_n+r-n_i\right|+\left|x_i'-\left(p-1\right)q_n+b_n\right|\geq q_n-2\left|b_n\right|$.
			\end{center}
			After absorbing factors into $o\left(1\right)$, we obtain
			\begin{center}
				$\left|\phi\left(n_i'\right)\right|e^{-\left|pq_n+r-n_i\right|L\left(E\right)}\leq e^{-\left(L\left(E\right)-o\left(1\right)\right)\left(2q_n\right)}r_{p-2}+e^{\left(L\left(E\right)-o\left(1\right)\right)q_n}r_{p-1}$.
			\end{center}
			Denoting $C_1=\max\left\{r_{p-2}r_{p-1}\right\}$, we get
			\begin{equation}
				\left|\phi\left(n_i'\right)\right|e^{-\left|pq_n+r-n_i\right|L\left(E\right)}\leq 2e^{\left(L\left(E\right)-o\left(1\right)\right)q_n}C_1.
			\end{equation}
			\item $n_i\in\left[\left(p-1\right)q_n,pq_n\right]$. Applying the same methods as in the previous case and using the fact that $\left|pq_n+r-n_i\right|+\left|n_i'-pq_n\right|\geq q_n-\left|b_n\right|$, one obtains
			\begin{equation}
				\left|\phi\left(n_i'\right)\right|e^{-\left|pq_n+r-n_i\right|L\left(E\right)}\leq e^{-\left(L\left(E\right)-o\left(1\right)\right)\left(q_n\right)}r_{p-1}+e^{\left(L\left(E\right)-o\left(1\right)\right)q_n}r_{p}.
			\end{equation}
		\end{enumerate}
	\end{enumerate}
	Let $M=\max\left\{r_k:k\in\left\{-2,-1,0,1,2\right\}\right\}$. Plugging all of this into (\ref{eq_1_qn+r}), we obtain
	\begin{equation}
		\left|\phi\left(pq_n+r\right)\right|\leq e^{-\left(L\left(E\right)-\left(1-t_1\right)\beta+o\left(1\right)\right)q_n}M.
	\end{equation}
	Now, $M=r_p$ leads to $r_p<r_p$ for large enough $n$ and so we must have $M\in\left\{r_{p-1},r_{p+1}\right\}$. This implies
	\begin{center}
		$r_p\leq\max\left\{r_{p\pm k}\exp\left(-\left(L\left(E\right)-\left(1-t_1\right)\beta-o\left(1\right)\right)q_n\right):k\in\left\{1,2\right\}\right\}$.
	\end{center}
	Note that for $p=0$ and sufficiently large $n$, the RHS of this inequality is arbitrarily small. On the other hand, by the normalization condition on $\phi$ (\ref{eq_norm_gen_ef}), $r_0\geq\frac{1}{2}$ and so the case $p=0$ is not allowed. Thus, we get (\ref{eq_r_j_estimate}), as required.
	
	Now, starting with $j=l$ and iterating iterating (\ref{eq_r_j_estimate}) $\frac{1}{2}\left(l-\floor{\frac{l}{q_n}}\right)$ times, we obtain
	\begin{center}
		$\left|\phi\left(k\right)\right|\leq r_l\leq r_qe^{-\frac{1}{2}\left(L\left(E\right)-\left(1-t_1\right)\beta-\o\left(1\right)\right)\left(l-\floor{\frac{l}{q_n}}\right)q_n}$.
	\end{center}
	Note that $q\leq 2l-\floor{\frac{l}{q_n}}$ and by (\ref{gen_ev_estimate_eq}) we  have
	\begin{center}
		$\left|\phi\left(k\right)\right|\leq C_E\left(1+2lq_n\right)e^{-\frac{1}{2}\left(L\left(E\right)-\left(1-t_1\right)\beta-o\left(1\right)\right)\left(l-\floor{\frac{l}{q_n}}\right)q_n}=$.
	\end{center}
	Now, choosing $o\left(1\right)$ correctly yields (\ref{eq_phi_k_estimate}).
\end{proof}
\begin{lemma}\label{prod_L_norms_growth_lemma}
	For $\mu$-almost every $E\in\mathbb{R}$, for every $\varepsilon>0$, for large enough $L>0$ we have
	\begin{equation}\label{est_on_prod_L_norms_1}
		\omega_+\left(L\right)=\underset{\theta}{\max}\|u_{\theta}\|_L^+\cdot\underset{\theta}{\min}\|u_{\theta}\|_L^+\geq L^{1+\frac{1}{2}\frac{L\left(E\right)}{t_1\beta}-\varepsilon}
	\end{equation}
	and
	\begin{equation}\label{est_on_prod_L_norms_2}
		\omega_-\left(L\right)=\underset{\theta}{\max}\|u_{\theta}\|_L^-\cdot\underset{\theta}{\min}\|u_{\theta}\|_L^-\geq L^{1+\frac{1}{2}\frac{L\left(E\right)}{t_1\beta}-\varepsilon}.
	\end{equation}
\end{lemma}
Lemma \ref{prod_L_norms_growth_lemma} is precisely \cite[Lemma 8.2]{JLT}, however the proof itself requires slight adaptations, again due to different decay estimates.
\begin{proof}
	We will prove (\ref{est_on_prod_L_norms_1}), the proof of (\ref{est_on_prod_L_norms_2}) is similar. Given $L>0$, let $n\in\mathbb{N}$ be such that $4q_n\leq L<4q_{n+1}$. Also, let $\theta\in\left[0,\pi\right)$ be such that
	\begin{center}
		$\|u_\theta\|_L^+=\underset{\theta'\in\left[0,\pi\right)}{\min}\|u_{\theta'}\|_L^+$.
	\end{center}
	Denote by $\phi$ the generalized eigenfunction corresponding with $E$. Finally, fix some \mbox{$D>1+\frac{1}{2}\frac{L\left(E\right)}{t_1\beta}-\varepsilon$}. We divide the proof into cases.
	
	\textbf{Case 1:} $q_{n+1}^{t_2}\leq 20q_n^2q_{n+1}^{t_1}$. In that case, under the notations of the previous section, we have $q_{n+1}\leq q_n^{C}$. By Proposition \ref{prop_diophantine_transition}, we get that for $y_1=q_n,y_2=q_n+1$ there exist $n_1,n_2\in\mathbb{Z}$ such that $y_1\in\left[n_1,n_1+q_n-1\right]$, $y_2\in\left[n_2,n_2+q_n-1\right]$, and for $i=1,2$ we have
	\begin{center}
		$\left|G_{\left[n_1,n_1+q_n-1\right]}\left(y_i,n_i\right)\right|,\left|G_{\left[n_1,n_1+q_n-1\right]}\left(y_i,n_i+q_n-1\right)\right|\leq e^{-\left(L\left(E\right)-o\left(1\right)\right)\frac{q_n}{2}}$.
	\end{center}
	By Proposition \ref{prop_green_omega_estimate}, we obtain
	\begin{center}
		$\omega_+\left(L\right)>\frac{e^{\frac{1}{2}\left(L\left(E\right)-o\left(1\right)\right)q_n}}{4}$.
	\end{center}
	For sufficiently large $n$ we have
	\begin{equation}
		\frac{e^{\frac{1}{2}\left(L\left(E\right)-o\left(1\right)\right)q_n}}{4}>4q_n^{DC}>4q_{n+1}^D>L^D>L^{1+\frac{1}{2}\frac{L\left(E\right)}{t_1\beta}-\varepsilon},
	\end{equation}
	which implies (\ref{est_on_prod_L_norms_1}).

	\textbf{Case 2:} $20q_n^2q_{n+1}^{t_1}<q_{n+1}^{t_2}$, $L>20q_n^2q_{n+1}^{t_1}$. In that case, there exist a pair of points \mbox{$k,k-1\in\left[3q_n^2q_{n+1}^{t_1}-q_n,3q_n^2q_{n+1}^{t_1}\right]$} which are $n$-resonant and satisfy $2q_n^2q_{n+1}^{t_1}\leq k,k-1\leq b_{n+1}$. By Proposition \ref{prop_liouville_resonant_growth}, we have
	\begin{center}
		$\left|\phi\left(k\right)\right|,\left|\phi\left(k-1\right)\right|\leq e^{-\frac{1}{2}\left(L\left(E\right)-\left(1-t_1\right)\beta-\varepsilon\right)k}$.
	\end{center}
	Let $\psi$ be the solution for which the Wronskian of $\psi$ and $\phi$ is $1$, namely
	\begin{center}
		$\phi\left(m\right)\psi\left(m+1\right)-\phi\left(m+1\right)\psi\left(m\right)=1$
	\end{center}
	for all $m\in\mathbb{Z}$. Then we have
	\begin{center}
		$\max\left\{\left|\psi\left(k\right)\right|,\left|\psi\left(k-1\right)\right|\right\}\geq\frac{1}{2}e^{\frac{1}{2}\left(L\left(E\right)-\left(1-t_1\right)\beta-\varepsilon\right)k}$
	\end{center}
	which in turn imlpies that $\|\psi\|_L^+\geq\frac{1}{2}e^{\frac{1}{2}\left(L\left(E\right)-\left(1-t_1\right)\beta-\varepsilon\right)k}$. In addition, since $\phi$ is normalized (\ref{eq_norm_gen_ef}), we have $\|\phi\|_L^+\geq 1$. Thus, we get
	\begin{center}
		$\omega_+\left(L\right)\geq\|\psi\|_L^+\geq \frac{1}{2}e^{\frac{1}{2}\left(L\left(E\right)-\left(1-t_1\right)\beta-\varepsilon\right)k}\geq\frac{1}{2}e^{\frac{1}{2}\left(L\left(E\right)-\left(1-t_1\right)\beta-\varepsilon\right)2q_n^2q_{n+1}^{t_1}}\geq\frac{1}{2}e^{L^{\frac{t_1}{2}}}$.
	\end{center}
	Taking $L$ large enough so that $e^{L}\geq \left(L^{1+\frac{1}{2}\frac{L\left(E\right)}{t_1\beta}-\varepsilon}\right)^{\frac{2}{t_1}}$ concludes the proof of Case 2.
	
	\textbf{Case 3:} $20q_n^2q_{n+1}^{t_1}<q_{n+1}^{t_2}$, $L\leq 20q_n^2q_{n+1}^{t_1}$, $q_n>L^\varepsilon$. In that case, let $k=\floor{\frac{q_n}{2}}$. By Corollary \ref{cor_growth_between_resonances},
	\begin{center}
		$\left|u_{\theta}\left(k\right)\right|,\left|u_\theta\left(k-1\right)\right|\leq e^{-\left(L\left(E\right)-\varepsilon\right)\left(1-\varepsilon\right)\frac{q_n}{4}}\left(r_0+r_1\right)\leq e^{-\left(L\left(E\right)-\varepsilon\right)\left(1-\varepsilon\right)\frac{q_n}{4}}r$,
	\end{center}
	where $r=\max\left\{r_0,r_1\right\}$. Now one can proceed exactly as in Case 1.
	
	\textbf{Case 4:} $20q_n^2q_{n+1}^{t_1}<q_{n+1}^{t_2}$, $L\leq 20q_n^2q_{n+1}^{t_1}$, $q_n\leq L^\varepsilon$. Let $p=\floor{\frac{L}{4q_n}}>1$. For $j=0,\ldots,p-1$, let $k_j=jq_n+\floor{\frac{q_n}{2}}$. Denote $\overline{r}_j=\max\left\{r_j,r_{j+1}\right\}$. By Corollary \ref{cor_growth_between_resonances}, we have
	\begin{center}
		$\left|u_\theta\left(k_j\right)\right|,\left|u_\theta\left(k_j-1\right)\right|\leq e^{-\left(L\left(E\right)-\varepsilon\right)\frac{q_n}{4}}\overline{r}_j$.
	\end{center}
	This implies that
	\begin{equation}\label{eq_case_4_1}
		\max\left\{\left|u_{\theta+\frac{1}{4}}\left(k_j\right)\right|,\left|u_{\theta+\frac{1}{4}}\left(k_j-1\right)\right|\right\}\geq\frac{1}{2\overline{r}_j}e^{\left(L\left(E\right)-\varepsilon\right)\frac{q_n}{4}}.
	\end{equation}
	Using (\ref{eq_case_4_1}) and the fact that $r_0,\ldots,r_p\in \left[0,L\right]$, for large enough $L$ we obtain
	\begin{center}
		$\sum\limits_{j=0}^{p-1} \overline{r}_j^2\leq2\sum\limits_{j=0}^p r_j^2\leq2\left(\|u_\theta\|_L^+\right)^2$,\\
		$\left(\|u_{\theta+\frac{1}{4}}\|_L^+\right)^2\geq\sum\limits_{j=0}^{p-1}\frac{e^{\left(L\left(E\right)-\varepsilon\right)\frac{q_n}{4}}}{2\overline{r}_j^2}$.
	\end{center}
	Now, we have
	\begin{center}
		$\left(\omega_+\left(L\right)\right)^2\geq\left(\|u_\theta\|_L^+\right)^2\left(\|u_{\theta+\frac{1}{4}}\|_L^+\right)^2\geq\frac{e^{\left(L\left(E\right)-\varepsilon\right)q_n}}{2}\left(\sum\limits_{j=1}^{p-1}\overline{r}_j^2\right)\left(\sum\limits_{j=1}^{p-1}\frac{1}{\overline{r}_j^2}\right)\geq\frac{e^{\left(L\left(E\right)-\varepsilon\right)q_n}}{2q_n^2}L^2$,
	\end{center}
	where the last inequality follows from Cauchy-Schwartz inequality. By the definition of $\limsup$, we have that for large enough $L$ (and consequently large enough $n$), $q_{n+1}\leq \left(e^{q_n}\right)^{\beta+\varepsilon}$. In addition, we may assume that $2q_n^2\leq L^\varepsilon$. Thus, we obtain
	\begin{center}
		$\frac{e^{\left(L\left(E\right)-\varepsilon\right)q_n}}{2q_n^2}L^2\geq e^{\left(L\left(E\right)-\varepsilon\right)q_n}L^{2-\varepsilon}=\left(\left(e^{q_n}\right)^{\beta+\varepsilon}
		\right)^{\frac{L\left(E\right)-\varepsilon}{\beta+\varepsilon}}\geq \left(q_{n+1}\right)^{\frac{L\left(E\right)}{\beta}}$.
	\end{center}
	Finally, using our assumption $L\leq 20q_n^2q_{n+1}^{t_1}$ which means that $q_{n+1}\geq\left(\frac{L}{20q_n^2}\right)^{\frac{1}{t_1}}$, and so
	\begin{center}
		$\left(q_{n+1}\right)^{\frac{L\left(E\right)}{\beta}}\geq \left(\left(\frac{L}{20q_n^2}\right)^{\frac{1}{t_1}}\right)^{\frac{L\left(E\right)}{\beta}}$
	\end{center}
	which concludes the proof of Case 4 and of Lemma \ref{prod_L_norms_growth_lemma}.
\end{proof}
Using Lemma \ref{prod_L_norms_growth_lemma}, the following can be obtained following the lines of \cite[Lemma 9.2]{JLT}. A proof can be found in the Appendix.
\begin{lemma}\label{hl_borel_lower_bound}
	For $\mu$-almost every $E\in\mathbb{R}$ there exists some $\theta\in\left[0,\pi\right)$ such that for every \mbox{$t\in\left(0,\frac{L\left(E\right)}{2\beta-L\left(E\right)}\right)$} there exists $\varepsilon_0>0$ such that for every $\varepsilon<\varepsilon_0$,
	\begin{equation}
		\im m_{\theta}^\pm\left(E+i\varepsilon\right)\geq\varepsilon^{-t}.
	\end{equation}
\end{lemma}
We can now prove Theorem \ref{main_thm}
\begin{proof}
	By Lemma \ref{hl_borel_lower_bound} and Proposition \ref{line_borel_hl_borel_prop}, we obtain that for every $t\in\left(0,\frac{L\left(E\right)}{2\beta-L\left(E\right)}\right)$ there exists $\varepsilon_0>0$ such that for every $0<\varepsilon<\varepsilon_0$,
	\begin{equation}\label{M_bd_eq}
		\im M\left(E+i\varepsilon\right)\geq\varepsilon^{-t}.
	\end{equation}
	If $L\left(E\right)\geq\beta\left(\alpha\right)$, then (\ref{M_bd_eq}) is true for every $t\in\left(0,1\right)$ and so we get that for every $\eta\in\left(0,1\right)$,
	\begin{center}
		$\underset{\varepsilon\to0}{\liminf}\,\varepsilon^{1-\eta}\im M\left(E+i\varepsilon\right)>0$,
	\end{center}
	which by Proposition \ref{packing_dim_borel_transform_prop} implies that $\gamma_\mu^+\left(E\right)=0$ and so by Lemma \ref{packing_derivative_infinity_lemma}, for every $\eta>0=2\left(1-\frac{\Lambda\left(E\right)}{\beta\left(\alpha\right)}\right)$, $\underline{D}_\mu^\eta\left(E\right)=\infty$, as required. Suppose now that $L\left(E\right)<\beta\left(\alpha\right)$. By (\ref{M_bd_eq}), we obtain that for every $\eta\in\left(\frac{2\beta-2L\left(E\right)}{2\beta-L\left(E\right)},1\right)$,
	\begin{equation}\label{M_bd_eq_2}
		\underset{\varepsilon\to0}{\liminf}\,\varepsilon^{1-\eta}\im M\left(E+i\varepsilon\right)>0.
	\end{equation}
	This implies, by Proposition \ref{packing_dim_borel_transform_prop}, that for every $\eta\in\left(\frac{2\beta-2L\left(E\right)}{2\beta-L\left(E\right)},1\right)$, $\gamma_\mu^+\left(E\right)\leq\frac{2\eta}{2-\eta}$. By the monotonicity of $x\to\frac{2x}{2-x}$, denoting $\eta_0=\frac{2\beta-2L\left(E\right)}{2\beta-L\left(E\right)}$ we obtain that
	\begin{center}
		$\gamma_\mu^+\left(E\right)\leq\frac{2\eta_0}{2-\eta_0}=\frac{2\beta-2L\left(E\right)}{\beta}=2\left(1-\frac{\Lambda\left(E\right)}{\beta\left(E\right)}\right)$.
	\end{center}
	Theorem \ref{main_thm} now follows from Lemma \ref{packing_derivative_infinity_lemma}.
\end{proof}

\appendix
\section{Proof of Lemma \ref{hl_borel_lower_bound}}
We prove Lemma \ref{hl_borel_lower_bound} following the lines of \cite[Lemma 9.2]{JLT}.
\begin{proof}[Proof of Lemma \ref{hl_borel_lower_bound}]
	We will prove that there exists $\theta\in\left[0,\pi\right)$ such that
	\begin{equation}\label{eq_appendix_0}
		\im m_{\theta}^+\left(E+i\varepsilon\right)\geq\varepsilon^{-t}.
	\end{equation}
	The proof for $m_\theta^-$ is similar.
	
	Fix $\eta>0$. Let $\theta\in\left[0,\pi\right)$ be such that $u_\theta$ (as defined in Section \ref{section_prelim}) satisfies (\ref{gen_ev_estimate_eq}) and (\ref{eq_norm_gen_ef}) (namely $u_\theta$ is the generalized eigenfunction). Following the notations of Section \ref{section_prelim}, given $L>0$, let
	\begin{center}
		$b_+\left(L\right)=\|u_\theta\|_L^+$,\\
		$\omega_+\left(L\right)\coloneqq\underset{\theta\in\left[0,\pi\right)}{\max}\|u_\theta\|_L^+\cdot\underset{\theta\in\left[0,\pi\right)}{\min}\|u_\theta\|_L^+$.
	\end{center}
	Recall that $u_\theta$ satisfies (\ref{eq_last_simon}) which implies
	\begin{equation}\label{eq_1_appendix}
		b_+\left(L\right)\leq C\left(E\right)L^{1+\eta}.
	\end{equation}
	Also let $L\left(\varepsilon\right)$ be such that $\omega_+\left(L\left(\varepsilon\right)\right)=\frac{1}{\varepsilon}$. By Lemma \ref{prod_L_norms_growth_lemma}, we have that for small enough $\varepsilon$ (and consequently large enough $L\left(\varepsilon\right))$,
	\begin{equation}
		\frac{1}{\varepsilon}=\omega_+\left(L\left(\varepsilon\right)\right)\geq L\left(\varepsilon\right)^{1+g},
	\end{equation}
	where $g=\frac{1}{2}\frac{L\left(E\right)}{t_1\beta}-\eta$. This implies that
	\begin{equation}\label{eq_2_appendix}
		L\left(\varepsilon\right)\leq\varepsilon^{-\frac{1}{1+g}}.
	\end{equation}
	By Theorem \ref{Appendix_thm}, (\ref{eq_1_appendix}) and (\ref{eq_2_appendix}) we obtain
	\begin{equation}
		\im m_\theta^+\left(E+i\varepsilon\right)\geq\frac{1}{C}\frac{1}{\varepsilon}\frac{1}{b_+\left(L\left(\varepsilon\right)\right)}\geq\varepsilon^{t_0},
	\end{equation}
	where
	\begin{center}
		$t_0=1-\frac{1+\eta}{1+g}=\frac{\beta-L\left(E\right)}{2\beta-L\left(E\right)+2\sigma\beta-\eta\beta-\eta L\left(E\right)}$.
	\end{center}
	Letting $\eta$ and $\sigma$ go to $0$, we will eventually have $t_0\in\left(t,\frac{L\left(E\right)}{2\beta-L\left(E\right)}\right)$, which implies (\ref{eq_appendix_0}), as required.
\end{proof}

\section{Proof of Lemma \ref{JK_Sec_6_Lemma}}\label{app_2}
We will prove Lemma \ref{JK_Sec_6_Lemma}, following the lines of \cite[Section 7.8]{JK}. We will use the following
\begin{theorem}\emph{\cite[Theorem 6.4]{JK}}\label{JK_thm_6.4}
	Let $n\in\mathbb{N}$, $x,m,h_0,\ldots,h_l\in\mathbb{Z}$. For every $j\in\left\{0,\ldots,l\right\}$, let $x_j\coloneqq x+h_j\alpha\mod 1$. Suppose that for every $i,j\in\left\{0,\ldots,l\right\}$, $\left|x_i-x_j\right|\geq d>0$. Let $\Lambda=\sigma\left(H_n\left(x\right)\right)$. Suppose that there exists an interval $K\subseteq\mathbb{R}$ and $\gamma>0$ such that $\left|K\right|<\left(10\gamma\right)^{-1}$, $x_0,\ldots,x_l\in K$ and
	\begin{center}
		$\Lambda\cap K\leq t\leq l$.
	\end{center}
	Finally, suppose that for every $z\in\Lambda\setminus K$, $\dist\left(z,K\right)>\gamma^{-1}$. Then there exists $j\in\left[0,l\right]$ such that for $J=\left[h_j,h_j+n\right]\coloneqq\left[n_1,n_2\right]$ and for every $m\in\mathbb{Z}$ which satisfies
		\begin{center}
		$h_j+\frac{n}{4}\leq m\leq h_j+\frac{3n}{4}$,
	\end{center}
	we have
	\begin{equation}\label{JK_thm_6.4_eq}
		\left|G_J\left(m,n_i\right)\right|\leq\frac{e^{-\left|m-n_i\right|L\left(E\right)\left(1-o\left(1\right)\right)}}{\gamma^t\left(\frac{d}{2}\right)^tt!}.
	\end{equation}
\end{theorem}
\begin{proof}[Proof of Lemma \ref{JK_Sec_6_Lemma}]
	Let $I=I_1\cup I_2$, and let $S=I\alpha\mod 1$. The following facts are proved in \cite[Section 7.8]{JK}:
	\begin{enumerate}
		\item $S$ is the union of $q_n$ disjoint intervals of length $\frac{2}{9q_n}$, each such interval contains exactly two points from $S$.
		\item The distance between two points in the same interval is $\left|j\right|\dist\left(q_n\alpha,\mathbb{Z}\right)$.
		\item The $\frac{1}{9q_n}$-neighborhoods of these intervals are disjoint.
	\end{enumerate}
	Let $\Lambda=\sigma\left(H_{2q_n-1}\right)$. Since $\left|\Lambda\right|\leq 2q_n-1$, there exists at least one interval $K$ such that $\left|K\cap\Lambda\right|\leq 1$. Note that this interval $K$ contains two points $x_0=p_0\alpha,x_1=p_1\alpha$ which satisfy the assumptions of Threom \ref{JK_thm_6.4} with $l=t=1$, $d=\left|j\right|\dist\left(q_n\alpha,\mathbb{Z}\right)$ and $\gamma=1$. Thus, we obtain some $j_0\in\left\{p_0,p_1\right\}$ such that for every $m\in\mathbb{Z}$, if $h_j+\frac{n}{5}\leq m\leq h_j+\frac{4n}{5}$, then (\ref{JK_thm_6.4_eq}) holds. Finally, it is shown in \cite[Section 7.8]{JK} that
	\begin{center}
		$\left|j\right|\dist\left(q_n\alpha,\mathbb{Z}\right)\geq e^{-\left(\beta+o\left(1\right)\right)q_n+\log\left|j\right|}$,
	\end{center}
	which concludes the proof.
\end{proof}

\begin{thebibliography}{1}
	
	\bibitem{AJ} A.~Avila and S.~Jitomirskaya, {\it The Ten Martini Problem}, Ann.\ of Math.\ \textbf{170} (2009), 303--342.
	
	\bibitem{AZY} A.~Avila, J.~You and Q.~Zhou, {\it Sharp phase transitions for the almost Mathieu operator}, Duke Math.\ J.\ \textbf{166} (2017), 2697--2718.
	
	\bibitem{Ber} Ju.~M.~Berezanski\u{i}, {\it Expansions in Eigenfunctions of Self-Adjoint Operators}, Transl.\ Math.\ Mono. \textbf{17}, Amer.\ Math.\ Soc., Providence, RI, 1968.
	
	\bibitem{CF} C.~Cedzich and J.~Fillman, {\it Absence of bound states for quantum walks and CMV matrices via reflections}, J.\ Spectr.\ Theory \textbf{14} (2024), 1513--1536.
	
	\bibitem{CFO} C.~Cedzich, J.~Fillman and D.~Ong, {\it Almost everything about the unitary almost Mathieu operator}, Commun.\ Math.\ Phys \textbf{403} (2023), 745--794.
	
	\bibitem{Cut} C.~D.~Cutler, {\it  Measure disintegrations with respect to $\sigma$-stable monotone indices and the pointwise representation of packing dimension}, Supp.\ Rend.\ Circ.\ Mat.\ Palermo \textbf{28} (1992), 319--339.
	
	\bibitem{DF1} D.~Damanik and J.~Fillman, {\it One-dimensional ergodic Schrödinger operators--I. General theory}, Grad.\ Stud.\ Math.\ \textbf{221}, American Mathematical Society, Providence, RI (2022).
	
	\bibitem{DF2} D.~Damanik and J.~Fillman, {\it One-dimensional ergodic Schrödinger operators--II. Specific classes}, Grad.\ Stud.\ Math.\ \textbf{221}, American Mathematical Society, Providence, RI (2025).
	
	\bibitem{DKL} D.~Damanik, R.~Killip and D.~Lenz, {\it Uniform Spectral Properties of One-Dimensional Quasicrystals, III. $\alpha$-continuity}, Commun.\ Math.\ Phys \textbf{212} (2000), 191--204.
	
	\bibitem{DT} D.~Damanik and S.~Tcheremchantsev, {\it Scaling estimates for solutions and dynamical lower bounds on wavepacket spreading}, J.\ Anal.\ Math \textbf{97} (2005), 103--131.
	
	\bibitem{DJLS} R.~Del Rio, S.~Jitomirskaya, Y.~Last and B.~Simon, {\it Operators with singular continuous spectrum, IV. Hausdorff dimensions, rank one perturbations, and localization}, J.\ Anal.\ Math \textbf{69} (1996), 153--200.
	
	\bibitem{GJ} L.~Ge and S.~Jitomirskaya, {\it Hidden subcriticality, symplectic structure, and universality of sharp arithmetic spectral results for type I operators}, preprint. arXiv:2407.08866.
	
	\bibitem{GP} D.~J.~Gilbert and D.~B.~Pearson, {\it On subordinacy and analysis of the spectrum of one-dimensional Schr\"{o}dinger operators}, J.\ Math.\ Anal.\ Appl. \textbf{128} (1987), 30--56.
	
	\bibitem{GSB} I.~Guarneri and H.~Schulz-Baldes, {\it Lower bounds on wave packet propagation by packing dimensions of spectral measures}, Math.\ Phys.\ Electron.\ J.\ \textbf{16} (1999), Paper 1.
	
	\bibitem{Jit0} S.~Jitomirskaya, {\it Anderson localization for the almost Mathieu equation: a nonperturbative proof}, Commun.\ Math.\ Phys.\ \textbf{165} (1994), 49--57.
	
	\bibitem{Jit} S.~Jitomirsaya, {\it Almost everything about the almost mathieu operator, II}, Proc.\ of XI Int.\ Congress of Math. Physics, Int.\ Press, Somerville, Mass.\ (1995), 373--382.
	
	\bibitem{Jit2} S.~Jitomirskaya, {\it  Metal-insulator transition for the almost Mathieu operator}, Ann.\ of Math.\ \textbf{150} (1999), 1159--1175. 
	
	\bibitem{JK1} S.~Jitomirskaya and I.~Kachkovskiy, {\it All couplings localization for quasiperiodic operators with monotone potentials}, J.\ Eur.\ Math.\ Soc.\ \textbf{21} (2019), 777--795.
	
	\bibitem{JK2} S.~Jitomirskaya and I.~Kachkovskiy, {\it Sharp arithmetic delocalization for quasiperiodic operators with potentials of semi-bounded variation}, preprint. arXiv:2408.16935.
	
	\bibitem{JK} S.~Jitomirskaya and I.~Kachkovskiy, {\it Sharp arithmetic localization for quasiperiodic operators with monotone potentials.}, preprint. 	arXiv:2407.00703.
	
	\bibitem{JL1} S.~Jitomirskaya and Y.~Last, {\it Power law subordinacy and singular spectra, I. Half-line operators}, Acta Math. \textbf{183} (1999), 171--189.
	
	\bibitem{JL2} S.~Jitomirskaya and Y.~Last, {\it Power law subordinacy and singular spectra, II. Line operators}, Commun.\ Math.\ Phys. \textbf{211} (2000), 643--658.
	
	\bibitem{JW} S.~Jitomirskaya and W.~Liu, {\it Universal hierarchical structure of quasiperiodic eigenfunctions}, Ann.\ of Math.\ \textbf{187} (2018), 721--776.
	
	\bibitem{JLT} S.~Jitomirskaya, W.~Liu and S.~Tcheremchantsev, {\it Lower bounds on concentration through Borel transform and quantitative singularity of spectral measures near the arithmetic transition}, preprint.
	
	\bibitem{JZ} S.~Jitomirskaya and S.~Zhang, {\it Quantitative continuity of singular continuous spectral measures and arithmetic criteria for quasiperiodic Schr\"{o}dinger operators}, J.\ Eur.\ Math.\ Soc.\ \textbf{24} (2022), 1723--1767.
	
	\bibitem{Kach} I.~Kachkovskiy, {\it Localization for quasiperiodic operators with unbounded monotone potentials}, J.\ Funct.\ Anal.\ \textbf{277} (2019), 3467--3490.
	
	\bibitem{Ker} J.~Kerdboon and X.~Zhu, {\it Anderson localization for Schr\"{o}dinger operators with monotone potentials over circle homeomorphisms}, J.\ Spectr.\ Theory \textbf{14} (2024), 1623--1646.
	
	\bibitem{KKL} R.~Killip, A.~Kiselev and Y.~Last, {\it Dynamical upper bounds on wavepacket spreading}, Amer.\ J.\ Math.\ \textbf{125} (2003), 1165--1198.
	
	\bibitem{KP} S.~Khan and D.~B.~Pearson, {\it Subordinacy and spectral theory for infinite matrices}, Helv.\ Phys.\ Acta \textbf{65} (1992), 505--527.
	
	\bibitem{LS} Y.~Last and B.~Simon, {\it Eigenfunctions, transfer matrices, and absolutely continuous spectrum of one-dimensional Schr\"{o}dinger operators}, Invent.\ Math.\ \textbf{135} (1999), 329--367.
	
	\bibitem{LY} W.~Liu and X.~Yuan, {\it Anderson localization for the almost Mathieu operator in the exponential regime}, J.\ Spectr.\ Theory \textbf{5} (2015), 89--112.
	
	\bibitem{LY1} W.~Liu and X.~Yuan, {\it Anderson localization for the completely resonant phases}, J.\ Funct.\ Anal.\ \textbf{268} (2015), 732--747.
	
	\bibitem{RT} C.~A.~Rogers and S.~J.~Taylor, {\it The analysis of additive set functions in Euclidean space}, Acta Math. \textbf{101} (1959), 273--302.
	
	\bibitem{Sim1} B.~Simon, {\it Spectral analysis of rank one perturbations and applications}, in “Proc. Mathematical Quantum Theory, II: Schr\"{o}dinger Operators” (Vancouver, Canada, 1993), CRM Proceedings and Lecture Notes, 8, American Mathematical Society, Providence, RI, 1995, 109–149.
	
	\bibitem{FY} F.~Yang, {\it Anderson localization for the unitary almost Mathieu operator}, Nonlinearity \textbf{37} (2024), 085010.

\end{thebibliography}
\end{document}